\documentclass[12pt,oneside]{amsart}

\usepackage[margin=1.1in]{geometry}
\usepackage{amsmath,amsfonts,amssymb,amsthm,mathtools}
\usepackage{microtype}
\usepackage[colorlinks=true,linkcolor=blue,citecolor=blue,urlcolor=blue]{hyperref}
\usepackage{graphicx}
\numberwithin{equation}{section}

\newtheorem{theorem}{Theorem}[section]
\newtheorem{proposition}[theorem]{Proposition}
\newtheorem{lemma}[theorem]{Lemma}
\newtheorem{corollary}[theorem]{Corollary}
\newtheorem{definition}[theorem]{Definition}

\newcommand{\R}{\mathbb R}
\newcommand{\Hq}{\mathbb H}
\newcommand{\SpOne}{\operatorname{Sp}(1)}
\newcommand{\SO}{\mathrm{SO}}
\newcommand{\SU}{\mathrm{SU}}
\newcommand{\wt}{\operatorname{wt}}
\newcommand{\len}[1]{|#1|}
\newcommand{\ii}{\mathbf i}
\newcommand{\jj}{\mathbf j}
\newcommand{\kk}{\mathbf k}

\title[Almost laws for $\SO(3)$]{Improved Almost laws for $\SO(3)$}
\author{Gal Yehuda}
\address{Department of Mathematics, Yale University, New Haven, CT, USA}
\email{gal.yehuda@yale.edu}
\date{}

\begin{document}

\begin{abstract}
We construct quantitative almost laws for $\SO(3)$.  More precisely, there
exist a constant $c>0$ and non-trivial words $W_n\in F_2$ such that, for every
$A,B\in \SO(3)$,
\[
  \|W_n(A,B)-I\|
  \le \exp\!\left(-c |W_n|^{\delta}\right),
\]
where $\delta=\log_2(x_0)=0.879146\ldots$ and $x_0>1$ is the real root of
$x^3=x^2+x+1$.  This improves the exponent $\log_2\varphi$ obtained from
Elkasapy's lower-central-series construction.
As an application, we show how this result improves the
word-length threshold in Kuperberg's Solovay--Kitaev algorithm for
single-qubit gates.
\end{abstract}

\maketitle

\section{Introduction}
Let $F_2=\langle a,b\rangle$ be the free group on two generators.  
If $G$ is a compact group equipped with a bi-invariant metric $d_G$, a non-trivial word $w\in F_2$ is called an $\varepsilon$-almost law for $G$ if
\[
  d_G\bigl(w(g,h),1_G\bigr)\le \varepsilon
  \qquad\text{for all } g,h\in G.
\]
The quantitative problem is to estimate the shortest possible length as
$\varepsilon\to0$.  Thom proved the qualitative existence of almost laws
\cite{Thom2010}; they also play a role in the work of Chen--Hurtado--Lee
\cite{ChenHurtadoLee2021}.  Quantitative constructions arise from short words
deep in the lower central series \cite{ElkasapyThom2013,Elkasapy2016}.
Elkasapy constructed words in $\gamma_n(F_2)$ of length
$O(n^{\log_\varphi 2})$, which yield the exponent
$\log_2\varphi=0.694\ldots$ for compact Lie groups.  We improve this exponent
for $\SO(3)$ by using its rank-one quaternionic geometry.

\subsection{Main result}

We equip $\SO(3)$ with the operator-norm metric inherited from $M_3(\mathbb R)$.

\begin{theorem}\label{thm:global}
There are constants $c,C>0$ and non-trivial words $W_n\in F_2$ such that
\[
  |W_n|\le C2^n
\]
and, for every pair of rotations $A,B\in \SO(3)$,
\[
  \|W_n(A,B)-I\|
  \le
  \exp\!\left(-c\,|W_n|^{\delta}\right),
\]
where $\delta = \log_2(x_0)$ and $x_0>1$ is the real root of
$x^3 = x^2 + x + 1$.

Consequently, if $L_{\SO(3)}(\varepsilon)$ denotes the least length of an $\varepsilon$-almost law for $\SO(3)$, then
\[
  L_{\SO(3)}(\varepsilon)
  \ll
  \bigl(\log(1/\varepsilon)\bigr)^{1/\delta}.
\]
\end{theorem}

\subsection{Outline of the construction}

The proof has a local and a global step.  For the local construction,
introduce free generators $p,q$ and set
\[
  u_0=p^2q^{-1}p^{-1}.
\]
Let $\tau$ and $\rho$ be the automorphisms defined by
\[
  \tau(p)=q,\qquad \tau(q)=p,
  \qquad
  \rho(p)=q^{-1},\qquad \rho(q)=p^{-1}.
\]
We define
\[
  u_{n+1}=
  \begin{cases}
    u_n\tau(u_n),& n\text{ even},\\
    u_n\rho(u_n),& n\text{ odd},
  \end{cases}
\]
and obtain words in the original generators by the substitution
\[
  w_n(a,b)=u_n(a,bab^{-1}).
\]
Each step uses two copies of the preceding word, and hence
$|w_n|=O(2^n)$.

To pass to arbitrary rotations, choose $\eta_0>0$ so that the local
estimate holds in the $\eta_0$-neighborhood of the identity, and fix a
non-trivial $\eta_0$-almost law $v$.  Put
\[
  b_k=a^kba^{-k},\qquad
  X=v(b_0,b_1),\qquad
  Y=v(b_2,b_3),
\]
and define
\[
  W_n=w_n(X,Y).
\]
For every $A,B\in\SO(3)$, both $X(A,B)$ and $Y(A,B)$ lie in the
neighborhood where the local estimate applies.  The elements
$b_0,b_1,b_2,b_3$ freely generate a free group, so $X$ and $Y$ freely generate a
rank-two free subgroup; in particular, $W_n$ is non-trivial.  Since $v$
is fixed, this substitution changes the word length only by a constant
factor.

\subsection{Geometric viewpoint}
It easier to understand the geometric idea of the construction when working with the unit quaternions. 
Write a quaternion as $r+x\ii+y\jj+z\kk$, where
$\ii^2=\jj^2=\kk^2=\ii\jj\kk=-1$.  The unit quaternions
\[
  \SpOne=\{q\in \Hq: |q|=1\}\simeq \SU(2)
\]
double-cover $\SO(3)$ by the action $v\mapsto qvq^{-1}$ on
$\operatorname{Im}\Hq\simeq \R^3$ \cite{Hall2015}.  For imaginary
quaternions $U,V$, multiplication takes the geometric form
\[
  UV=-U\cdot V+U\times V.
\]
Thus $[U,V]=2U\times V$.

The substitution
\[
  p=a,\qquad q=bab^{-1}
\]
produces two rotations with the same angle, whose axes are carried to one
another by $b$.  In the unit quaternions, write
\[
  p=c+P,\qquad q=c+Q,
  \qquad P,Q\in\operatorname{Im}\Hq.
\]
Since $p$ and $q$ are conjugate, $|P|=|Q|$, and therefore
\[
  (P+Q)\cdot(P-Q)=0.
\]
Choosing $\ii$ and $\jj$ in the directions of $P+Q$ and $P-Q$ gives
\[
  p=c+s\ii+t\jj,\qquad
  q=c+s\ii-t\jj,\qquad
  c^2+s^2+t^2=1.
\]
Thus the $\ii$-direction bisects the two axes, while the $\jj$-direction
measures their separation, see Figure \ref{fig:conjugate-pair}.

\begin{figure}[t]
  \centering
  \includegraphics[width=0.85\textwidth]{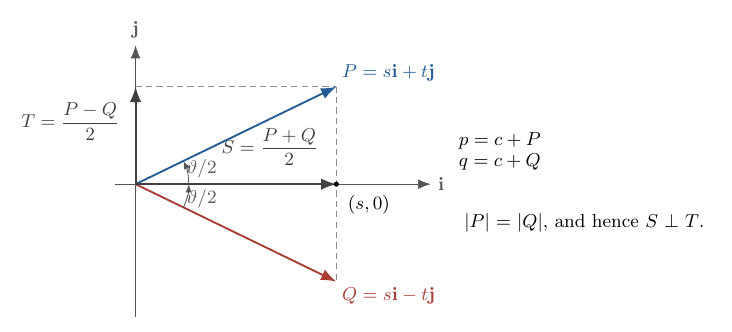}
  \caption{The imaginary parts $P$ and $Q$ of the conjugate pair have
  equal length. Their average $S=(P+Q)/2$ and half-difference
  $T=(P-Q)/2$ are orthogonal, giving the normal form
  $P=s\ii+t\jj$ and $Q=s\ii-t\jj$.}
  \label{fig:conjugate-pair}
\end{figure}

When $a$ and $b$ are close to the identity,
this separation combines the small angle of $a$ with the small tilt
produced by $b$.

In these coordinates, $\tau$ and $\rho$ are induced by half-turns about
the $\ii$- and $\jj$-axes, see Figure \ref{fig:half-turn}.
If $u=r+v$ and $r+v'$ is either half-turned
copy, then
\[
  \operatorname{Im}\bigl((r+v)(r+v')\bigr)
  =r(v+v')+v\times v'.
\]
The two components reversed by the half-turn cancel in $v+v'$ and can
reappear only through the cross product.  Alternating the two half-turns
repeats this cancellation in different coordinate directions.  The
resulting calculation gives the equation
$x^3=x^2+x+1$; since the word length doubles at each step, this is the
geometric origin of the exponent $\log_2(x_0)$.

\begin{figure}[t]
  \centering
  \includegraphics[width=.65\textwidth]{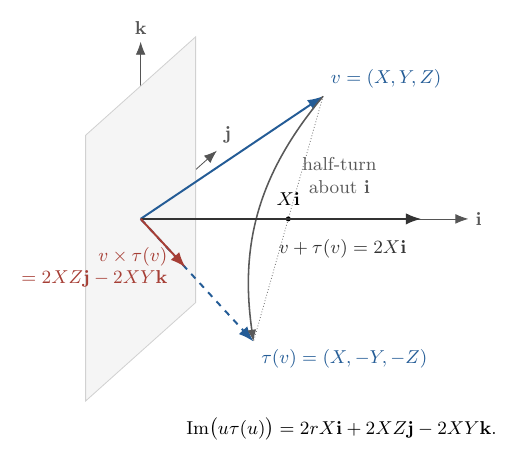}
  \caption{The automorphism $\tau$ acts on the imaginary part as a
  half-turn about the $\ii$-axis. The transverse components cancel in
  $v+\tau(v)$ and reappear only through the cross product, giving
  $\operatorname{Im}(u\tau(u))
  =2rX\ii+2XZ\jj-2XY\kk$.}
  \label{fig:half-turn}
\end{figure}

\section{conjugate-pair normal form}\label{sec:normal-form}

Let $p,q\in \SpOne$ be conjugate and near $1$.  Write
\[
  p=c+u,
  \qquad q=c+v,
\]
with $u,v\in \operatorname{Im}\Hq$.  Since $p$ and $q$ are conjugate, they have the same real part and hence
\[
  |u|=|v|.
\]
Set
\[
  S=\frac{u+v}{2},
  \qquad
  T=\frac{u-v}{2}.
\]
Then
\[
  S\cdot T=\frac{|u|^2-|v|^2}{4}=0.
\]
Thus, after rotating the imaginary quaternion basis, every nearby conjugate pair has the form
\begin{equation}\label{eq:pq-normal}
  p=c+s\ii+t\jj,
  \qquad
  q=c+s\ii-t\jj,
  \qquad
  c=\sqrt{1-s^2-t^2}.
\end{equation}

Now return to the original generators $a,b$ and set
\[
  p=a,
  \qquad q=bab^{-1}.
\]
If
\[
  a=e^A,
  \qquad b=e^B,
\]
with $A,B$ small imaginary quaternions, then
\[
  \operatorname{Im}(p)=A+O\bigl((\|A\|+\|B\|)^3\bigr),
  \qquad
  \operatorname{Im}(q)=A+[B,A]+O\bigl((\|A\|+\|B\|)^3\bigr).
\]
Therefore, in the above decomposition, the average direction satisfies
\[
  S=A+\frac12[B,A]+O\bigl((\|A\|+\|B\|)^3\bigr),
\]
while the difference direction satisfies
\[
  T=-\frac12[B,A]+O\bigl((\|A\|+\|B\|)^3\bigr).
\]
Consequently, if $\varepsilon=\|A\|+\|B\|$, then
\[
  |s|=O(\varepsilon),\qquad |t|=O(\varepsilon^2).
\]
This motivates the following definition.

\begin{definition}
For a monomial in $s,t$, define
\[
  \wt(s^\alpha t^\beta)=\alpha+2\beta.
\]
The weighted valuation of a nonzero formal power series is the least weight
of a monomial that occurs in it; we set $\wt(0)=\infty$.
\end{definition}

For $u=r+X\ii+Y\jj+Z\kk$, its \emph{weighted order} is the
minimum of the weighted valuations of $X,Y,Z$.  This order controls the
ordinary Taylor expansion even though the normalizing frame above depends on
the point.  Indeed, weighted order at least $d$ implies
$\|\operatorname{Im}u\|=O(\varepsilon^d)$, independently of the chosen
orthonormal frame.  Since $u$ is a unit quaternion with scalar part $1$ at
the origin, its scalar displacement is $O(\varepsilon^{2d})$.  Hence
\begin{equation}\label{eq:weighted-to-ordinary}
  u(a,bab^{-1})=1+O(\varepsilon^d).
\end{equation}
As the word map is analytic, all of its Taylor terms of degree less than $d$
vanish.

\section{sign symmetries}\label{sec:symmetries}

Let a word $u(p,q)$ evaluate, in the normal form \eqref{eq:pq-normal}, as
\[
  u=r+X\ii+Y\jj+Z\kk,
\]
where $r,X,Y,Z\in \R[[s,t]]$.
Define
\[
  \nu(u)=\bigl(\wt(X),\wt(Y),\wt(Z)\bigr),
\]

There are two exact involutions of the free group $\langle p,q\rangle$ which act by sign changes on the components.

\begin{lemma}\label{lem:signs}
Let
\[
  \tau:p\leftrightarrow q,
\]
\[
  \rho:p\mapsto q^{-1},\qquad q\mapsto p^{-1}.
\]
Then, in the normal form \eqref{eq:pq-normal},
\[
  \tau(u)=r+X\ii-Y\jj-Z\kk,
  \qquad
  \rho(u)=r-X\ii+Y\jj-Z\kk.
\]
\end{lemma}

\begin{proof}
Conjugation by $\ii$ sends
\[
  \ii(c+s\ii+t\jj)\ii^{-1}=c+s\ii-t\jj,
\]
so it exchanges $p$ and $q$.  Thus $\tau$ is realized by conjugation by $\ii$, which fixes the $\ii$-component and changes the signs of the $\jj$- and $\kk$-components.  Similarly, conjugation by $\jj$ sends
\[
  p\mapsto q^{-1},\qquad q\mapsto p^{-1}.
\]
This gives the stated sign rules.
\end{proof}

The multiplication rules that drive the construction are now elementary quaternion arithmetic.

\begin{lemma}\label{lem:multiplication}
Let $u=r+X\ii+Y\jj+Z\kk$.  Then
\begin{align*}
  \operatorname{Im}(u\tau(u))
    &=(2rX)\ii+(2XZ)\jj-(2XY)\kk,\\
  \operatorname{Im}(u\rho(u))
    &=(-2YZ)\ii+(2rY)\jj+(2XY)\kk.
\end{align*}
Consequently, if $\nu(u)=(A,B,C)$, then
\[
  \nu(u\tau(u))=(A,A+C,A+B),
  \qquad
  \nu(u\rho(u))=(B+C,B,A+B).
\]
\end{lemma}

\begin{proof}
Write $v=(X,Y,Z)$ for the imaginary vector of $u$.  The imaginary part of a product $(r+v)(r+v')$ is
\[
  r(v+v')+v\times v'.
\]
For $\tau(u)$ we have $v'=(X,-Y,-Z)$.  Hence
\[
  r(v+v')=(2rX,0,0),
  \qquad
  v\times v'=(0,2XZ,-2XY).
\]
This proves the formula for $u\tau(u)$.  The formula for $u\rho(u)$ follows similarly from $v'=(-X,Y,-Z)$:
\[
  r(v+v')=(0,2rY,0),
  \qquad
  v\times v'=(-2YZ,0,2XY).
\]
Since $r$ has constant term $1$ and the coefficient ring is an integral
domain, the weighted valuations are as claimed.
\end{proof}

\section{The seed word}\label{sec:seed}

The seed is
\[
  u_0=p^2q^{-1}p^{-1}.
\]
It is chosen so that its three component valuations provide the ``initial data'' for the recursion.

\begin{lemma}\label{lem:seed}
For
\[
  p=c+s\ii+t\jj,
  \qquad q=c+s\ii-t\jj,
\]
one has
\[
  u_0=r_0+X_0\ii+Y_0\jj+Z_0\kk
\]
with
\begin{align*}
  X_0&=8cst^2,\\
  Y_0&=2ct(1-4s^2),\\
  Z_0&=2st(3-4s^2-4t^2).
\end{align*}
In particular
\[
  \nu(u_0)=(5,2,3).
\]
\end{lemma}

\begin{proof}
Using $p^{-1}=c-s\ii-t\jj$, $q^{-1}=c-s\ii+t\jj$, and
$\ii\jj=\kk$, direct multiplication gives the displayed expressions.  Their
initial weighted monomials are respectively $st^2,t,st$, of weights $5,2,3$.
\end{proof}

\section{The recursion}\label{sec:recursion}

Define
\[
  u_{n+1}=\begin{cases}
  u_n\tau(u_n),& n\text{ even},\\
  u_n\rho(u_n),& n\text{ odd}.
  \end{cases}
\]
Let
\[
  d_{-1}=1,
  \qquad d_0=2,
  \qquad d_1=5,
\]
and for $m\ge -1$ define
\[
  d_{m+3}=d_{m+2}+d_{m+1}+d_m.
\]

\begin{theorem}\label{thm:local}
Let $x_0>1$ be the real root of
\[
  x^3=x^2+x+1.
\]
There exists an explicit sequence of non-trivial words
$w_n\in F_2=\langle a,b\rangle$ such that
\[
  \len{w_n}\le C2^n
\]
and such that the local expansion of $w_n$ at the identity in $\SpOne\times\SpOne$ has no non-constant term before degree $d_n$, where
\[
  d_0=2,\qquad d_1=5,\qquad d_2=8,
\]
and
\[
  d_{n+3}=d_{n+2}+d_{n+1}+d_n.
\]
In particular,
\[
  d_n\ge c\,\len{w_n}^{\delta},
  \qquad \delta=\log_2x_0.
\]
\end{theorem}

\begin{proposition}\label{prop:components}
For every $k\ge0$,
\[
  \nu(u_{2k})=\bigl(d_{2k+1},\ d_{2k},\ d_{2k}+d_{2k-1}\bigr),
\]
and
\[
  \nu(u_{2k+1})=\bigl(d_{2k+1},\ d_{2k+2},\ d_{2k+1}+d_{2k}\bigr).
\]
Therefore the minimum of the three weighted component valuations of $u_n$ is $d_n$.
\end{proposition}

\begin{proof}
The formula for $k=0$ in the even case is exactly Lemma \ref{lem:seed}:
\[
  \nu(u_0)=(5,2,3)=(d_1,d_0,d_0+d_{-1}).
\]
Assume the even formula for $u_{2k}$.  Applying Lemma \ref{lem:multiplication} to $u_{2k}\tau(u_{2k})$ gives
\[
  \nu(u_{2k+1})=
  \bigl(d_{2k+1},\ d_{2k+1}+d_{2k}+d_{2k-1},\ d_{2k+1}+d_{2k}\bigr).
\]
By the recurrence,
\[
  d_{2k+2}=d_{2k+1}+d_{2k}+d_{2k-1},
\]
so the odd formula follows.

Conversely, assume the odd formula for $u_{2k+1}$.  Applying Lemma \ref{lem:multiplication} to $u_{2k+1}\rho(u_{2k+1})$ gives
\[
  \nu(u_{2k+2})=
  \bigl(d_{2k+2}+d_{2k+1}+d_{2k},\ d_{2k+2},\ d_{2k+2}+d_{2k+1}\bigr).
\]
Again by the recurrence,
\[
  d_{2k+3}=d_{2k+2}+d_{2k+1}+d_{2k},
\]
which proves the even formula at $k+1$.

Since the sequence $d_m$ is strictly increasing for $m\ge0$, the minimum of the three entries in $\nu(u_n)$ is $d_n$.
\end{proof}

\begin{proof}[Proof of Theorem \ref{thm:local}]
Let
\[
  w_n(a,b)=u_n(a,bab^{-1}).
\]
Equation~\eqref{eq:weighted-to-ordinary} and
Proposition~\ref{prop:components} show that the Taylor expansion of $w_n$ on
$\SpOne\times\SpOne$ has no non-constant term before degree $d_n$.

The same proposition shows that $u_n$ has a nonzero imaginary component on
the normal-form family.  Every pair in that family consists of conjugate unit
quaternions, so it has the form $p=a$, $q=bab^{-1}$ for some $a,b\in\SpOne$.
Thus the word map $w_n$ is not identically $1$, and hence $w_n\ne1$ in $F_2$.

The length satisfies
\[
  \len{u_n}_{p,q}\le 2^n\len{u_0}=4\cdot 2^n,
\]
because $\tau$ and $\rho$ preserve length and each step uses two copies of the previous word.  Substituting $p=a$ and $q=bab^{-1}$ gives
\[
  \len{w_n}_{a,b}\le 3\len{u_n}_{p,q}\le 12\cdot 2^n.
\]
Finally, $d_n$ grows like $x_{0}^n$ where $x_0$ is the real root of
\[
  x^3 = x^2 + x + 1.
\]
Thus
\[
  d_n\ge cx_0^n\ge c'\len{w_n}^{\log_2x_0}.
\]
This proves the theorem.
\end{proof}

\section{From local estimates to global almost laws}\label{sec:global}

\begin{lemma}[Uniform local estimate]\label{lem:uniform-local}
Let $w_n(a,b)=u_n(a,bab^{-1})$ be the local words of Theorem \ref{thm:local}.  There are constants $\eta_0>0$ and $c_0>0$ such that, whenever $x,y\in \SO(3)$ satisfy
\[
  \|x-I\|\le \eta_0,
  \qquad
  \|y-I\|\le \eta_0,
\]
one has
\[
  \|w_n(x,y)-I\|\le \exp(-c_0 d_n),
\]
where $d_n$ is the sequence appearing in Theorem \ref{thm:local}.
\end{lemma}

\begin{proof}
We first prove a uniform estimate in the quaternionic normal form
\[
  p=c+s\ii+t\jj,\qquad q=c+s\ii-t\jj,
  \qquad c=\sqrt{1-s^2-t^2}.
\]
Write
\[
  u_n=r_n+X_n\ii+Y_n\jj+Z_n\kk,
  \qquad
  \nu(u_n)=(A_n,B_n,C_n).
\]
By Proposition \ref{prop:components},
\[
  \min(A_n,B_n,C_n)=d_n.
\]

Define auxiliary triples
\[
  \ell_n=(\alpha_n,\beta_n,\gamma_n)
\]
by
\[
  \ell_0=(1,1,1),
\]
and, according as $n$ is even or odd,
\[
  \ell_{n+1}
  =
  (\alpha_n+1,\alpha_n+\gamma_n+1,\alpha_n+\beta_n+1),
\]
or
\[
  \ell_{n+1}
  =
  (\beta_n+\gamma_n+1,\beta_n+1,\alpha_n+\beta_n+1).
\]
A direct induction from $\nu(u_0)=(5,2,3)$ shows that
\[
  \ell_n\le \nu(u_n)
\]
componentwise.  Indeed, subtracting these recurrences from the valuation recurrences in
Lemma \ref{lem:multiplication}, the error vector transforms at even steps by
\[
  (\Delta_X,\Delta_Y,\Delta_Z)
  \mapsto
  (\Delta_X-1,\Delta_X+\Delta_Z-1,\Delta_X+\Delta_Y-1),
\]
and at odd steps by
\[
  (\Delta_X,\Delta_Y,\Delta_Z)
  \mapsto
  (\Delta_Y+\Delta_Z-1,\Delta_Y-1,\Delta_X+\Delta_Y-1).
\]
If $n$ is even, the invariant
\[
  (\Delta_X,\Delta_Y,\Delta_Z)\ge(4,1,2)
\]
is carried to $(3,5,4)$ or larger; at the following odd step this is
carried to $(8,4,7)\ge(4,1,2)$.  This proves the claim by induction.

We claim that there is an absolute constant $K$, independent of $n$, such that for
\[
  |s|\le \theta,\qquad |t|\le \theta^2,\qquad \theta\le \frac12,
\]
one has
\[
  |X_n|\le K^{\alpha_n}\theta^{A_n},\qquad
  |Y_n|\le K^{\beta_n}\theta^{B_n},\qquad
  |Z_n|\le K^{\gamma_n}\theta^{C_n}.
\]
For $n=0$ this is Lemma~\ref{lem:seed}, after fixing $K$ large enough.
Lemma~\ref{lem:multiplication} gives the induction step: since
$u_n\in\SpOne$, we have $|r_n|\le1$, and each factor $2$ is absorbed by
the corresponding ``$+1$'' in the definition of $\ell_{n+1}$.

Since $\ell_n\le \nu(u_n)$ componentwise, we obtain
\[
  |X_n|\le (K\theta)^{A_n},\qquad
  |Y_n|\le (K\theta)^{B_n},\qquad
  |Z_n|\le (K\theta)^{C_n}.
\]
Hence
\[
  |X_n|+|Y_n|+|Z_n|
  \le 3(K\theta)^{d_n}.
\]
Choose $\theta_0\le1/2$ so small that $K\theta_0<1/4$.  Since
$d_n\ge2$, the fixed factor $3$ can be absorbed into the exponent, and
there is $c_1>0$ such that
\[
  |X_n|+|Y_n|+|Z_n|
  \le \exp(-c_1d_n)
\]
for all $\theta\le \theta_0$ and all $n$.  Shrink $\theta_0$ so that the
left-hand side is also less than $1/2$.  The parameter box is connected,
$u_n(0,0)=1$, and the scalar part cannot change sign without passing through
zero.  Hence, uniformly in $n$,
\[
  r_n=\sqrt{1-X_n^2-Y_n^2-Z_n^2},
\]
and therefore
\[
  |r_n-1|\le 2(X_n^2+Y_n^2+Z_n^2).
\]
After decreasing $c_1$ once more, we get
\[
  |u_n(p,q)-1|\le \exp(-c_1d_n)
\]
throughout the normal-form region
\[
  |s|\le \theta,\qquad |t|\le \theta^2,\qquad \theta\le\theta_0.
\]

Now let $x,y\in \SO(3)$ be close to the identity and let $a,b\in\SpOne$
be their near-identity lifts.  Section~\ref{sec:normal-form}, together with
the local bi-Lipschitz equivalence of the two models, gives constants
$C_s,C_t$ such that the normal-form parameters of
$p=a$, $q=bab^{-1}$ satisfy
\[
  |s|\le C_s\eta,\qquad |t|\le C_t\eta^2,
  \qquad
  \eta=\max\{\|x-I\|,\|y-I\|\}.
\]
Choose $D\ge\max\{C_s,\sqrt{C_t},1\}$ and put $\theta=D\eta$.
For $\eta_0\le\theta_0/D$, the normal-form estimate applies.  The covering
map $\SpOne\to \SO(3)$ is locally Lipschitz, so its fixed multiplicative
constant can be absorbed by decreasing the exponential constant.  Thus
\[
  \|w_n(x,y)-I\|\le \exp(-c_0d_n)
\]
for some $c_0>0$ independent of $n,x,y$.
\end{proof}

\begin{proof}[Proof of Theorem \ref{thm:global}]
By \cite[Proposition~3.7]{Thom2010}, choose a fixed non-trivial
$\eta_0$-almost law $v\in F_2$ for $\SO(3)$, where $\eta_0$ is as in
Lemma~\ref{lem:uniform-local}.  Thus
\[
  \|v(A,B)-I\|\le \eta_0
  \qquad\text{for all }A,B\in \SO(3).
\]
We evaluate $v$ in two disjoint free factors.  Let
\[
  x_1=b,\qquad x_2=aba^{-1},\qquad
  y_1=a^2ba^{-2},\qquad y_2=a^3ba^{-3}.
\]
Schreier's method shows that the conjugates $a^kba^{-k}$ form a free basis
for the kernel of
\[
  F_2\to\mathbb Z,\qquad a\mapsto 1,\quad b\mapsto 0.
\]
In particular, if $H=\langle x_1,x_2,y_1,y_2\rangle$, then
\[
  H=\langle x_1,x_2\rangle *\langle y_1,y_2\rangle.
\]
Put
\[
  X(a,b)=v(x_1,x_2),\qquad Y(a,b)=v(y_1,y_2),
\]
and define
\[
  W_n(a,b)=w_n(X(a,b),Y(a,b)).
\]
For every $A,B\in \SO(3)$, the two arguments fed into $w_n$ are both $\eta_0$-close to the identity, because $v$ is an $\eta_0$-almost law and $x_i(A,B),y_i(A,B)$ are rotations.  Lemma \ref{lem:uniform-local} gives
\[
  \|W_n(A,B)-I\|\le \exp(-c_0 d_n).
\]

The substitution has fixed cost: if $m=\len{v}$, then
\[
  \len{X},\len{Y}\le 7m,
  \qquad
  \len{W_n}\le 7m\,\len{w_n}.
\]
The two substitutions of $F_2$ into the free factors of $H$ are injective,
so $X,Y\ne1$.  The free-product normal form then shows that $X,Y$ freely
generate a rank-two free subgroup.  Since $w_n\ne1$, it follows that
\[
  W_n=w_n(X,Y)\neq 1.
\]

By Theorem \ref{thm:local},
\[
  d_n\ge c\len{w_n}^{\delta}.
\]
Since
\[
  \len{W_n}\le 7m\,\len{w_n},
\]
we get
\[
  d_n\ge c(7m)^{-\delta}\len{W_n}^{\delta}.
\]
Therefore
\[
  \|W_n(A,B)-I\|
  \le
  \exp\bigl(-c'\len{W_n}^{\delta}\bigr)
\]
for a constant $c'>0$ independent of $n,A,B$.  This proves the stated
quantitative almost-law bound, as well as $|W_n|\ll2^n$.

It remains to justify the estimate for $L_{\SO(3)}(\varepsilon)$.  Put
$R=\log(1/\varepsilon)$ and choose $n$ minimal with $c_0d_n\ge R$.  Since
$d_n\le3d_{n-1}$ for $n\ge1$, minimality gives $d_n\ll R$.  Moreover,
$d_n\asymp x_0^n$, and hence
\[
  |W_n|\ll2^n
  \ll d_n^{\log_{x_0}2}
  \ll R^{1/\delta}.
\]
This completes the proof.
\end{proof}

\section{Application to single-qubit compilation}

For a word $\omega\in F_2$, let
$\operatorname{ccan}_{\SU(2)}(\omega)$ be its \emph{conjugate cancellation
degree} in the sense of Kuperberg: the largest $r$ such that
\[
  \omega(e^{tA},e^{tB})=\exp(O(t^r))
\]
whenever $A,B\in\mathfrak{su}(2)$ are conjugate.  A fixed word of length
$L$ and conjugate cancellation degree $r\ge2$ can be used in Kuperberg's
construction with every word-length exponent
$\alpha>\max\{1,\log_r L\}$ \cite[Sec.~4.2]{Kuperberg2023}.

\begin{corollary}[Single-qubit compilation]\label{cor:kuperberg}
Let $\mathcal A=\mathcal A^{-1}\subset \SU(2)$ be finite, and suppose that
$\langle\mathcal A\rangle$ is dense in $\SU(2)$.  For every
\[
  \kappa>\log_{x_0}2=\frac1\delta=1.137466951\ldots,
\]
there is a polynomial-time classical algorithm which, given $g\in \SU(2)$ and
$N\in\mathbb N$, produces an $\mathcal A$-word $V$ such that
\[
  d(V,g)<2^{-N},\qquad |V|=O(N^\kappa).
\]
Here $d$ is any fixed bi-invariant Riemannian metric on $\SU(2)$.
Equivalently, precision $\varepsilon$ is achieved with
$O((\log(1/\varepsilon))^\kappa)$ gates.  Thus, for single-qubit gates, the
threshold $\log_\varphi2=1.440420\ldots$ in \cite{Kuperberg2023} can be
replaced by $\log_{x_0}2$.  The same conclusion holds projectively for
$\mathrm{PU}(2)\simeq \SO(3)$.
\end{corollary}

\begin{proof}
Set $L_n=|w_n|$ and
$r_n=\operatorname{ccan}_{\SU(2)}(w_n)$.  The local vanishing in
Theorem~\ref{thm:local} gives $r_n\ge d_n$.  Since $w_n$ is non-trivial,
$r_n<\infty$ by \cite[Thm.~5.4]{Kuperberg2023}.  Therefore
\[
  \limsup_{n\to\infty}\frac{\log L_n}{\log r_n}
  \le \frac{\log2}{\log x_0}=\log_{x_0}2.
\]
Given $\kappa>\log_{x_0}2$, choose $n$ such that
$\log_{r_n}L_n<\kappa$, and then choose
\[
  \max\{1,\log_{r_n}L_n\}<\alpha<\kappa.
\]
Using $w_n$ as the fixed cancellation word in Kuperberg's Algorithms SB and
ZB, the length recurrence in equation~(4.8) and Lemma~4.3 of
\cite{Kuperberg2023} give $|V|=O(N^\alpha)$, and hence the stated bound.
\end{proof}

\section*{Acknowledgments}
We thank Itay Glazer for useful feedback on a previous version of the manuscript.

\bibliographystyle{alpha}
\bibliography{so3_tribonacci_paper_iter5}

@article{Thom2010,
  author        = {Thom, Andreas},
  title         = {Convergent sequences in discrete groups},
  journal       = {Canadian Mathematical Bulletin},
  volume        = {56},
  number        = {2},
  pages         = {424--433},
  year          = {2013},
  eprint        = {1003.4093},
  archivePrefix = {arXiv},
  primaryClass  = {math.GR},
  doi           = {10.4153/CMB-2011-155-3}
}

@article{ChenHurtadoLee2021,
  author  = {Chen, Lvzhou and Hurtado, Sebastian and Lee, Homin},
  title   = {A height gap in {$GL_d(\overline{\mathbb Q})$} and almost laws},
  journal = {Groups Geom. Dyn.},
  volume  = {19},
  number  = {3},
  pages   = {899--912},
  year    = {2025},
  doi     = {10.4171/GGD/800}
}

@article{ElkasapyThom2013,
  author  = {Elkasapy, Abdelrhman and Thom, Andreas},
  title   = {On the length of the shortest non-trivial element in the derived and the lower central series},
  journal = {J. Group Theory},
  volume  = {18},
  number  = {5},
  pages   = {793--804},
  year    = {2015},
  doi     = {10.1515/jgth-2015-0007}
}

@misc{Elkasapy2016,
  author        = {Elkasapy, Abdelrhman},
  title         = {A new construction for the shortest non-trivial element in the lower central series},
  year          = {2016},
  eprint        = {1610.09725},
  archivePrefix = {arXiv},
  primaryClass  = {math.GR}
}

@book{Hall2015,
  author    = {Hall, Brian C.},
  title     = {Lie Groups, Lie Algebras, and Representations: An Elementary Introduction},
  edition   = {2},
  publisher = {Springer},
  series    = {Graduate Texts in Mathematics},
  volume    = {222},
  year      = {2015},
  doi       = {10.1007/978-3-319-13467-3}
}

@misc{Kuperberg2023,
  author        = {Kuperberg, Greg},
  title         = {Breaking the cubic barrier in the {Solovay--Kitaev} algorithm},
  year          = {2023},
  eprint        = {2306.13158},
  archivePrefix = {arXiv},
  primaryClass  = {quant-ph},
  doi           = {10.48550/arXiv.2306.13158},
  note          = {Revised version, 2025}
}

\end{document}